\documentclass[a4paper,11pt]{article}
\pdfoutput=1

\usepackage[left = 2.5cm, right = 2.5cm, top = 2.5cm, bottom = 2.5cm, headsep = 12pt, headheight = 15pt]{geometry}
\usepackage{parskip}

\usepackage[utf8]{inputenc}
\usepackage{amsfonts,amsmath,amssymb,mathtools,microtype,anyfontsize}
\usepackage[noBBpl]{mathpazo} % Palatino
\usepackage[usenames, dvipsnames]{xcolor} % Colours with names -- see https://www.overleaf.com/learn/latex/Using_colours_{i}n_LaTeX for list
\usepackage[hyphens]{url} % Urls together with line breaking them
\usepackage[linktoc = all, hidelinks]{hyperref} % Must be loaded after url
\hypersetup{colorlinks,linkcolor={blue!60!black},urlcolor={blue!60!black}, citecolor={green!60!black}}

\usepackage{booktabs,caption,enumitem,float,graphicx,mdwlist,subcaption} % Tables, figures & lists
\setlist[itemize]{topsep=0ex,itemsep=0ex,parsep=0.4ex}
\setlist[enumerate]{topsep=0ex,itemsep=0ex,parsep=0.4ex}
\usepackage{amsthm,thmtools,thm-restate} % Theorem environments
\usepackage[capitalise, compress, nameinlink, noabbrev]{cleveref} % Must be loaded after hyperref
\declaretheorem[name = Theorem, style = plain]{theorem} %numberwithin = section, 

\declaretheorem[name = Conjecture, numberlike = theorem, style = plain]{conjecture}

\declaretheorem[name = Lemma, numberlike = theorem, style = plain]{lemma}

\usepackage[sort,compress,numbers]{natbib}
\renewcommand{\epsilon}{\varepsilon}

\renewcommand{\geq}{\geqslant}
\renewcommand{\leq}{\leqslant}
\renewcommand{\emptyset}{\varnothing}

\DeclarePairedDelimiter{\abs}{\lvert}{\rvert} 
\DeclarePairedDelimiter{\floor}{\lfloor}{\rfloor} 
\DeclarePairedDelimiter{\set}{\{}{\}}
\DeclarePairedDelimiter{\parentheses}{(}{)}
\DeclareMathOperator{\Bin}{Bin}

\newcommand{\defn}[1]{\textcolor{Maroon}{\emph{#1}}}

\newcommand*{\EX}{\mathbb{E}}
\newcommand*{\OO}{\mathcal{O}}
\newcommand*{\PP}{\mathcal{P}}
\newcommand*{\prob}{\mathbb{P}}
\newcommand*{\vv}[1]{\boldsymbol{#1}}

\begin{document}
	
	\author{Ant\'onio Gir\~ao\footnotemark[2]
		\qquad Freddie Illingworth\footnotemark[2]
		\qquad Alex Scott\footnotemark[2]
		\qquad David R. Wood\footnotemark[3]}
	
	\title{\bf Defective Colouring of Hypergraphs}
	
	\maketitle
	
	\begin{abstract}
		We prove that the vertices of every $(r + 1)$-uniform hypergraph with maximum degree $\Delta$ may be coloured with $c(\frac{\Delta}{d + 1})^{1/r}$ colours such that each vertex is in at most $d$ monochromatic edges. This result, which is best possible up to the value of the constant $c$, generalises the classical result of Erd\H{o}s and Lov\'asz who proved the $d = 0$ case.  
	\end{abstract}
	
	\renewcommand{\thefootnote}{\fnsymbol{footnote}} % Make affiliation marks symbols
	
	\footnotetext[2]{Mathematical Institute, University of Oxford, United Kingdom (\texttt{\{girao,illingworth,scott\}@maths.ox.ac.uk}). Research supported by EPSRC grant EP/V007327/1.}
	
	\footnotetext[3]{School of Mathematics, Monash University, Melbourne, Australia  (\texttt{david.wood@monash.edu}). Research supported by the Australian Research Council and a Visiting Research Fellowship of Merton College, University of Oxford.}
	
	\renewcommand{\thefootnote}{\arabic{footnote}} % Returns to numbered footnotes
	
	\section{Introduction}\label{sec:intro}
	
	Hypergraph colouring is a widely studied field with numerous deep results~\citep{SS21,CM17,CM16,CM15,FM11,FM08,FM13,BFM10,KKL12,KK09,LP22}. %CF96,CDM14,DLR95,LC11
	In a seminal contribution, \citet{EL75} proved that every $(r + 1)$-uniform hypergraph with maximum degree $\Delta$ has a vertex-colouring with at most $c\Delta^{1/r}$ colours and with no monochromatic edge, where $c$ is an absolute constant. The proof is a simple application of what is now called the Lov\'{a}sz local lemma, introduced in the same paper. Indeed, hypergraph colouring was the motivation for the development of the Lov\'{a}sz local lemma, which has become a staple of probabilistic combinatorics.  
	
	This paper proves the following generalisation of the result of \citet{EL75}. Here a vertex-colouring of a hypergraph is \defn{$d$-defective} if each vertex is in at most $d$ monochromatic edges. The case $d = 0$ corresponds to the result of \citet{EL75} mentioned above. 
	\begin{theorem}
		\label{main}
		For all integers $r\geq 1$ and $d\geq 0$ and $\Delta\geq \max\set{d + 1, 50^{100r^4}}$, every $(r + 1)$-uniform hypergraph $G$ with maximum degree at most $\Delta$ has a $d$-defective $k$-colouring, where 
		\begin{equation*}
			k \leq  100 \parentheses*{\frac{\Delta}{d + 1}}^{1/r}.
		\end{equation*}
	\end{theorem}
	Several notes on \cref{main} are in order.
	
	\begin{itemize}
		\item The bound on the number of colours in the theorem of \citet{EL75} and in 
		\cref{main} is best possible (up to the multiplicative constant) because of complete hypergraphs\footnote{Let $G$ be the $(r + 1)$-uniform complete hypergraph on $n$ vertices, which has maximum degree $\Delta = \binom{n - 1}{r}\leq (\frac{en}{r})^r$. In any $d$-defective  $k$-colouring of $G$, at least $\frac{n}{k}$ vertices are monochromatic, implying $d\geq \binom{n/k - 1}{r} > (\frac{n}{2kr})^r \geq \frac{\Delta}{(2ek)^r}$. 
			Thus $k \geq \frac{1}{2e}(\frac{\Delta}{d})^{1/r} $, which is within a constant factor of the upper bound in \cref{main}.}. It remains tight even for $(r + 1)$-uniform hypergraphs with no complete $(r + 2)$-vertex subhypergraph. For example, a hypergraph construction by \citet[\S~3.2.2]{CM17} has this property\footnote{Let $\vv{e_i}$ denote the $r$-dimensional vector with 1 in the $i\textsuperscript{th}$ coordinate and 0 elsewhere. Let $G$ be the $(r + 1)$-uniform hypergraph with vertex set $\{1, \dotsc, n\}^r$ and whose edges are $\{\vv{v}, \vv{v_1}, \dotsc, \vv{v_r}\}$ where, for each $i$, $\vv{v_i} - \vv{v}$ is a positive multiple of $\vv{e_i}$. Any $r + 2$ vertices induce at most two edges, so $G$ has contains no $(r + 2)$-clique. $G$ has maximum degree $\Delta = (n - 1)^r < n^r$. Suppose that $V(G)$ is coloured with $k \leq (\Delta/(d + 1))^{1/r}/r < n/(r(d + 1)^{1/r})$ colours. Then there is a monochromatic set $S \subseteq V(G)$ of size at least $(d + 1)^{1/r} r n^{r - 1}$. Apply the following iterative deletion procedure to $S$: if, for some coordinate $j$ and integers $a_1, \dotsc, a_{j - 1}, a_{j + 1}, \dotsc, a_r \in \{1, \dotsc, n\}$, there are less than $(d + 1)^{1/r}$ vertices in $S$ whose $i\textsuperscript{th}$ coordinate is $a_i$ for all $i \neq j$, then delete all these vertices. Let $S'$ be the set remaining after applying all such deletions. Each step deletes less than $(d + 1)^{1/r}$ vertices and at most $rn^{r - 1}$ steps occur so $S'$ is non-empty. Let $\vv{v} \in S'$ have the smallest coordinate sum. By definition of $S'$, for each $i$, there are at least $(d + 1)^{1/r}$ vertices $\vv{v_i} \in S'$ with $\vv{v_i} - \vv{v}$ being a positive multiple of $\vv{e_i}$. Hence, $\vv{v}$ has degree at least $d + 1$ in $S'$. Therefore, every $d$-defective colouring of $G$ uses more than $(\Delta/(d + 1))^{1/r}/r$ colours.}.
		%%%%%%%%%%%%%%%
		\item The assumption $\Delta \geq d + 1$ in \cref{main} is reasonable, since if $\Delta \leq d$ then one colour suffices. The assumption that $\Delta \geq 50^{100r^4}$ enables the uniform constant 100 in the bound on $k$. Of course, one could drop the assumption and replace 100 by some constant $c_r$ depending on $r$.
		%%%%%%%%%%%%%%%
		\item The graph ($r = 1$) case of \cref{main}  (with $\floor{\frac{\Delta}{d + 1}}+1$ colours) is easily proved by a max-cut type argument (see \cref{maxcut}) as was first observed by \citet{Lovasz66}. See \citep{WoodSurvey} for a comprehensive survey on defective graph colouring.
		%%%%%%%%%%%%%%%
		\item If $G$ is a linear hypergraph (that is, any two edges intersect in at most one vertex), then \cref{main} may be proved directly with the Lov\'{a}sz local lemma. Non-linear hypergraphs are hard because the number of neighbours of a vertex $v$ is not precisely determined by the degree of $v$. See the start of \cref{sec:proof} for details.
		
		\item \cref{main} can be rephrased as saying that for any $k$, $G$ has a $k$-colouring with maximum monochromatic degree $\OO(\frac{\Delta}{k^r})$ for fixed $r$. This is similar to a result of \citet{BS02} who showed that for any $k$ every $(r + 1)$-uniform hypergraph with $m$ edges has a $k$-colouring with $\OO(\frac{m}{k^r})$ monochromatic edges of each colour. In this light, \cref{main} is a variant on so-called judicious partitions~\citep{HWY16,ZTY15,Haslegrave14,BS10,XuYu09,Scott05,BS04,ABKS03,BS02}.  %YanXu08,FHZ14,HZ17,LLS16,Li18
		
	\end{itemize}

	\subsection{Notation}\label{sec:notation}
	
	Let $G$ be a hypergraph, which consists of a finite  vertex-set $V(G)$ and an edge-set $E(G)\subseteq 2^{V(G)}$. Let \defn{$e(G)$} $\coloneqq \abs{E(G)}$. $G$ is \defn{$r$-uniform} if every edge has size $r$. The \defn{link hypergraph} of a vertex $v$ in $G$, denoted \defn{$G_v$}, is the hypergraph with vertex-set $V(G)\setminus\set{v}$ 
	and edge-set $\{e\subseteq V(G)\setminus\{v\} \colon e \cup \set{v} \in E(G)\}$. If $G$ is $(r + 1)$-uniform, then $G_{v}$ is $r$-uniform. The \defn{degree} of a set of vertices $S\subseteq V(G)$, denoted \defn{$\deg(S)$}, is the number of edges in $G$ that contain $S$. We often omit set parentheses, so $\deg(x)$ and $\deg(u, v)$ denote the number of edges containing $x$ and the number of edges containing both $u$ and $v$, respectively. Let $\textcolor{Maroon}{\Delta(G)} \coloneqq \max\{\deg(v) \colon v \in V(G)\}$.
	
	\subsection{Probabilistic Tools}\label{tools}
	
	We use the following standard probabilistic tools. 
	
	\begin{lemma}[Lov\'{a}sz local lemma~\citep{EL75}] \label{LLL}
		Let $\mathcal{A}$ be a set of events in a probability space such that each event in $\mathcal{A}$ occurs with probability at most $p$ and for each event $A \in \mathcal{A}$ there is a collection $\mathcal{A}'$ of at most $d$ other events such that $A$ is independent from the collection $(B \colon B \not\in \mathcal{A}' \cup \set{A})$. If $4pd \leq 1$, then with positive probability no event in $\mathcal{A}$ occurs. 
	\end{lemma}
	
	\begin{lemma}[Markov's  inequality]\label{Markov}
		If $X$ is a nonnegative random variable and $a>0$, then 
		\begin{equation*}
			\prob(X\geq a) \leq \frac {\EX(X)}{a}.
		\end{equation*}
	\end{lemma}
	
	\begin{lemma}[Chernoff bound]\label{Chernoff}
		Let $X \sim \Bin(n, p)$. For any $\epsilon \in [0, 1]$,
		\begin{align*}
			\prob(X \geq (1 + \epsilon)\EX(X)) & \leq \exp(-\epsilon^2 np/3), \\
			\prob(X \leq (1 - \epsilon)\EX(X)) & \leq \exp(-\epsilon^2 np/2).
		\end{align*}
	\end{lemma}

	We will need a version of Chernoff for negatively correlated random variables, for example, see \cite[Thm.~1]{IK10}. Boolean random variables $X_1, \dotsc, X_n$ are \defn{negatively correlated} if, for all $S \subseteq \{1, \dotsc, n\}$,
	\begin{equation*}
		\prob(X_i = 1 \textnormal{ for all } i \in S) \leq \prod_{i \in S} \prob(X_i = 1).
	\end{equation*}
	
	\begin{lemma}[Chernoff for negatively correlated variables]\label{Chernoffneg}
		Suppose $X_1, \dotsc, X_n$ are negatively correlated Boolean random variables with $\prob(X_i = 1) \leq p$ for all $i$. Then, for any $t \geq 0$,
		\begin{equation*}
			\prob\bigl(\sum_i X_i \geq pn + t\bigr) \leq \exp(-2t^2/n).
		\end{equation*}
	\end{lemma}
	
	Finally we need McDiarmid's bounded differences inequality~\citep{McDiarmid1989}.
	
	\begin{lemma}[McDiarmid's inequality]\label{McDiarmid}
		Let $T_1, \dotsc, T_n$ be $n$ independent random variables. Let $X$ be a random variable determined by $T_1, \dotsc, T_n$, such that changing the value of $T_j$ \textnormal{(}while fixing the other $T_i$\textnormal{)} changes the value of $X$ by at most $c_j$. Then, for any $t \geq 0$,
		\begin{equation*}
			\prob(X \geq \EX(X) + t) \leq \exp\Bigl(-\frac{2 t^2}{\sum_{i} c_{i}^2}\Bigr).
		\end{equation*}
	\end{lemma}
	
	\section{Proof}\label{sec:proof}
	
	For motivation we first consider a na\"{i}ve application of the Lov\'{a}sz local lemma. Suppose $G$ is a linear $(r + 1)$-uniform hypergraph. Colour $G$ with $k \coloneqq \floor{100(\Delta/(d + 1))^{1/r}}$ colours uniformly at random. For each set $F$ of $d + 1$ edges all containing a common vertex, let $B_F$ be the event that the vertex set of $F$ is monochromatic. Then, since $G$ is linear, $p \coloneqq \prob(B_F) = k^{-r(d + 1)}$. For a fixed $F$, the number of $F'$ sharing a vertex with $F$ is at most $D \coloneqq (r(d + 1) + 1) \Delta (r + 1) \binom{\Delta}{d}$; here we have specified the vertex shared with $F$, the edge containing that vertex, the common vertex of the edges in $F'$, and finally the remaining $d$ edges of $F'$. Now $D \leq 3r^2 d \Delta (e\Delta/d)^d$ and so
	\begin{align*}
		4pD & \leq 4 \cdot 100^{-r(d + 1)} \bigl(\tfrac{d + 1}{\Delta}\bigr)^{d + 1} \cdot 3r^2 e^d d^{-d + 1} \Delta^{d + 1} \\
		& = 12 r^2 e^d \cdot 100^{-r(d + 1)} \cdot d(d + 1) \bigl(\tfrac{d + 1}{d}\bigr)^d \\
		& \leq 24 r^2 d^2 e^{d + 1} \cdot 100^{-r(d + 1)} \leq 1.
	\end{align*}
	Hence, by the Lov\'{a}sz local lemma, there is a colouring in which no $B_F$ occurs; that is, there is a $d$-defective $k$-colouring of $G$. It was crucial in this argument that $G$ was linear so that the powers of $\Delta$ in $D$ and $p$ cancelled out exactly. For non-linear $G$, the number of neighbours of a vertex $v$ is not determined by the degree of $v$ and so $p$ may be larger without a corresponding decrease in $D$. A more involved argument is required.
	
	\subsection{First Steps}
	
	Here we outline our colouring strategy before diving into the details. We are given an $(r + 1)$-uniform hypergraph $G$ with maximum degree $\Delta$ and wish to colour its vertices so that every vertex is in at most $d$ monochromatic edges. For a fixed colouring $\phi$, the \defn{monochromatic degree} of a vertex $v$, denoted \defn{$\deg_{\phi}(v)$}, is the number of monochromatic edges containing $v$ (which must have colour $\phi(v)$).
	
	First we colour the vertices of $G$ uniformly at random with $k$ colours where $k = \floor{49(\frac{\Delta}{d + 1})^{1/r}}$. Say a vertex is \defn{bad} if its monochromatic degree is greater than $d$ and \defn{good} otherwise. We are aiming for a colouring in which every vertex is good. The expected monochromatic degree of a vertex $v$ in such a colouring is $k^{-r} \deg(v) \leq k^{-r} \Delta \leq 49^{-r} (d + 1)$. In particular, each individual vertex has small (certainly, by Markov's inequality, at most $49^{-r}$) probability of being bad. However, the goodness of a vertex  $v$ depends on the colours assigned to vertices in the neighbourhood of $v$ and so $49^{-r}$ is not a sufficiently small probability to conclude (by, say, the Lov\'asz local lemma) that there is a particular colouring for which all vertices are good.
	
	Instead of colouring all of $G$ with a single-uniform colouring, we do so over many rounds. After a round (where we coloured a hypergraph $G$), any good vertices will keep their colours and be discarded (they have been coloured appropriately). Let $G'$ be the subhypergraph of $G$ induced by the bad vertices. In the next round we uniformly and randomly colour the vertices of $G'$ with a new palette of colours completely disjoint from those used in previous rounds. If the palettes all have the same size and the process runs for too many rounds, then we will end up using too many colours. However, if $\Delta(G') \leq 2^{-r} \Delta(G)$, then we can use half the number of colours in the next round and so use $\OO((\frac{\Delta}{d + 1})^{1/r})$ colours across all the rounds. Thus, our aim is to prove the following nibble-style lemma from which \cref{main} easily follows.

	\begin{lemma}\label{nibble}
		Fix non-negative integers $r,\Delta,d$ with $r \geq 1$ and $\Delta \geq \max\set{d + 1, 50^{50r^3}}$. Then every $(r + 1)$-uniform hypergraph $G$ with maximum degree at most $\Delta$ has a partial colouring with at most $49(\frac{\Delta}{d + 1})^{1/r} $ colours such that every coloured vertex has monochromatic degree at most $d$ and the subhypergraph $G'$ of $G$ induced by the uncoloured vertices satisfies
		$\Delta(G') \leq 2^{-r} \Delta$.
	\end{lemma}
	
	\begin{proof}[Proof of \cref{main} assuming \cref{nibble}]
		We start with a $(r + 1)$-uniform hypergraph $G$ with maximum degree at most $\Delta = \Delta_{0}$ for some $\Delta_{0} \geq \max\set{d + 1, 50^{100r^4}}$. Apply \cref{nibble} to get a partial colouring of $G$ where:
		\begin{itemize}
			\item every vertex has monochromatic degree at most $d$,
			\item at most $49(\frac{\Delta_0}{d + 1})^{1/r}$ colours are used, and
			\item the subhypergraph $G_1$ of $G$ induced by uncoloured vertices has $\Delta(G_1) \leq \Delta_1 = 2^{-r} \Delta_0$.
		\end{itemize}
		Iterate this procedure (using a palette of new colours each round) to obtain, for $i=0,1,\dots$, an induced subhypergraph $G_i$ of $G$ with $\Delta(G_i) \leq \Delta_i = 2^{-ri} \Delta$ such that $G[V(G) - V(G_i)]$ has been coloured with at most
		\begin{equation*}
			49 \bigl(\tfrac{\Delta_0}{d + 1}\bigr)^{1/r} + 49 \bigl(\tfrac{\Delta_1}{d + 1}\bigr)^{1/r} + \dotsb + 49 \bigl(\tfrac{\Delta_{i - 1}}{d + 1}\bigr)^{1/r} = 49 \bigl(\tfrac{\Delta}{d + 1}\bigr)^{1/r} (1 + 2^{-1} + \dotsb + 2^{-(i - 1)}) \leq 98 \bigl(\tfrac{\Delta}{d + 1}\bigr)^{1/r}
		\end{equation*}
		colours and every monochromatic degree is at most $d$. Continue carrying out rounds of colouring until $\Delta_i < d + 1$ or $\Delta_i < 50^{50r^3}$.
		
		First suppose that $\Delta_i < d + 1$ and so $\Delta(G_i) \leq d$. Use a single new colour on the entirety of $G_i$ to give a $d$-defective colouring of $G$. Now suppose that $d + 1 \leq \Delta_i < 50^{50r^3}$. Properly colour $G_i$ with $\Delta(G_i) + 1 \leq 50^{50r^3}$ colours. This gives a $d$-defective colouring of $G$ with at most
		\begin{equation*}
			98 \bigl(\tfrac{\Delta}{d + 1}\bigr)^{1/r} + 50^{50r^3} \leq 100 \bigl(\tfrac{\Delta}{d + 1}\bigr)^{1/r}
		\end{equation*}
		colours. The final inequality uses the fact that $\Delta \geq 50^{100r^4}$ and $d + 1 < 50^{50r^3}$.
	\end{proof}
	
	Recall that a vertex is bad for a colouring $\phi$ if it has monochromatic degree at least $d + 1$. Say that an edge $e$ is \defn{bad} for a colouring $\phi$ if every vertex in $e$ is bad (note that a bad edge is not necessarily monochromatic). Furthermore, say that a vertex is \defn{terrible} for a colouring $\phi$ if it is incident to more than $2^{-r} \Delta$ bad edges. \Cref{nibble} says that there is some colouring for which no vertex is terrible. The key to the proof of \cref{nibble} is to show that a vertex is terrible with low probability.
	
	In the remainder of the paper, we use the definitions of good, bad, and terrible given above and also set $k \coloneqq \floor{49(\frac{\Delta}{d + 1})^{1/r}}$.
	
	\begin{lemma}\label{EachVertexTerrible}
		Let $\Delta \geq 50^{50r^3}$. Let $G$ be an $(r + 1)$-uniform hypergraph with maximum degree at most $\Delta$.  In a uniformly random $k$-colouring of $V(G)$, each vertex $v$ of $G$ is terrible with probability at most $\Delta^{-5}$. \end{lemma}
	
	\begin{proof}[Proof of \cref{nibble} assuming \cref{EachVertexTerrible}]
		Randomly and independently assign each vertex of $G$ one of $k$ colours.
		For each vertex $v$, let $A_v$ be the event that $v$ is terrible. By \cref{EachVertexTerrible}, $\prob(A_v)\leq \Delta^{-5}$. The event $A_v$ depends solely on the colours assigned to vertices in the closed second neighbourhood of $v$. Thus if two vertices $v$ and $w$ are at distance at least 5 in $G$, then $A_v$ and $A_w$ are independent. Thus each event $A_v$ is mutually independent of all but at most $2(r\Delta)^4$ other events $A_w$. Since $4 \Delta^{-5} \cdot 2(r\Delta)^4 = 8 r^4 /\Delta \leq 1$, by the Lov\'{a}sz local lemma, with positive probability, no event $A_v$ occurs. Thus, there exists a $k$-colouring $\phi$ of $G$ such that no vertex is terrible. Let $G'$ be the subgraph of $G$ induced by the bad vertices. Since no vertex is terrible, $\Delta(G')\leq 2^{-r}\Delta$. Uncolour all the bad vertices: every coloured vertex is good and so has monochromatic degree at most $d$.
	\end{proof}
	
	It remains to prove \cref{EachVertexTerrible}, which we do in \cref{FinalProof}. We have now reduced the question to a local property of a random $k$-colouring. 
	
	A vertex $v$ is terrible if it is bad and at least $2^{-r} \Delta$ edges in its link graph, $G_v$, are bad. Analysing the dependence between the badness of different edges in $G_v$ is difficult. We sidestep this issue by using a decomposition into sunflowers.
	
	A \defn{sunflower with $p$ petals} is a collection $A_1, \dotsc, A_p$ of sets for which $A_1 \setminus K, \dotsc, A_p \setminus K$ are pairwise disjoint where $K \coloneqq A_1 \cap \dotsb \cap A_p$ (that is, $A_i \cap A_j = K$ for all distinct $i, j$). $K$ is the \defn{core} of the sunflower and $A_1 \setminus K, \dotsc, A_p \setminus K$ are its \defn{petals}.
	
	If $A_1, \dotsc, A_p$ are distinct edges of a uniform hypergraph that form a sunflower, then the petals are all non-empty with the same size. The core may be empty in which case the sunflower is a matching of size $p$. In a random colouring, the colourings on different petals of a sunflower are independent. Hence, it will be useful to partition the edges of hypergraphs into sunflowers with many petals together with a few edges left over.
	
	\begin{lemma}[Sunflower decomposition]\label{sunflower}
		Let $H$ be an $r$-uniform hypergraph and $a$ be a positive integer. There are edge-disjoint subhypergraphs $H_1, \dotsc, H_s$ of $H$ such that:
		\begin{itemize}
			\item Each $H_i$ is a sunflower with exactly $a$ petals.
			\item $H' = H - (E(H_1) \cup \dotsb \cup E(H_s))$ has fewer than $(ra)^r$ edges.
		\end{itemize}
	\end{lemma}
	
	\begin{proof}
		Let $H_1, \dotsc, H_s$ be a maximal collection of edge-disjoint subhypergraphs of $H$ where each $H_i$ is a sunflower with exactly $a$ petals. So  $H'$ contains no sunflower with $a$ petals. By the Erd\H{o}s-Rado sunflower lemma~\citep{ER60}, $e(H') \leq r! (a - 1)^r < (ra)^r$ (see \citep{ALWZ21,BCW21,Rao2020} for recent improved bounds in the sunflower lemma).
	\end{proof}
	
	The proof of \cref{EachVertexTerrible} uses a sunflower decomposition to show that if a vertex is terrible, then some reasonably large set of vertices $S$ must have at least $3^{-r}$ proportion of its vertices being bad. As noted above, each vertex is bad with probability at most $49^{-r}$ and so we expect $49^{-r} \abs{S}$ bad vertices in $S$. We are able to show that the number of bad vertices in (a suitable) $S$ is not much more than the expected number with very small failure probability. This is accomplished in  \cref{Slemmalargek,Slemmasmallk} below, 
	which correspond respectively to the case of large and small $k$.
	
	\subsection{When \texorpdfstring{$k$}{k} is large: \texorpdfstring{$k \geq \Delta^{1/(6r^2)}$}{}}\label{sec:largek}
	
	Recall that $k = \floor{49(\frac{\Delta}{d + 1})^{1/r}}$ throughout. When $k$ is large we expect a medium-sized vertex-set $S$ to have close to $\abs{S}$ different colours appearing on it (that is, to be close to rainbow). If two vertices have different colours, then the events that they are bad will be negatively correlated and hence we expect only a small proportion of $S$ to be bad. The negative correlation is made precise in \cref{Chernoffforbad} and the upper tail concentration of the number of bad vertices in $S$ is established in \cref{Slemmalargek}.
	
	\begin{lemma}\label{Chernoffforbad}
		Let $S = \set{v_1, \dotsc, v_{\ell}}$ be a set of at most $k$ vertices in $G$ and let $D$ be the event that $v_1, \dotsc, v_{\ell}$ are all given different colours. Let $X$ be the number of bad vertices in $S$. Then, in a uniformly random $k$-colouring of $V(G)$, for any $t \geq 0$,
		\begin{equation*}
			\prob(X \geq \ell \cdot 49^{-r} + t \mid D) \leq \exp\bigl(-2t^2/\ell\bigr).
		\end{equation*}
	\end{lemma}
	
	\begin{proof}
		Let $B_j$ be the event $\set{v_j \textnormal{ is bad}}$ and $X_j$ be the indicator random variable for $B_j$ so $X = \sum_{j} X_j$. For an edge $e$ containing a vertex $v$, the probability $e$ is monochromatic is $k^{-r}$. Hence, the expected number of monochromatic edges containing $v$ is at most $\Delta k^{-r}$. Thus, $\prob(X_j = 1) \leq 49^{-r}$ by Markov's inequality (\cref{Markov}). 
		
		Fix distinct colours $c_1, \dotsc, c_{\ell}$ and let $V_j$ be the set of vertices given colour $c_j$. Conditioned on the event $C_j = \set{v_j \textnormal{ is coloured } c_j}$, $B_j$ is increasing in $V_j$, while $D$ is non-increasing in $V_j$. Hence, by the Harris inequality~\citep{Harris1960}, $\prob(B_j \cap D \mid C_j) \leq \prob(B_j \mid C_j) \prob(D \mid C_j)$. Using this and the symmetry of the colours gives
		\begin{equation*}
			\prob(B_j \mid D) = \prob(B_j \mid D \cap C_j) = \frac{\prob(B_j \cap D \mid C_j)}{\prob(D \mid C_j)} \leq \prob(B_j \mid C_j) = \prob(B_j).
		\end{equation*}
		But $\prob(B_j) \leq 49^{-r}$, so $\EX(X \mid D) = \sum_j \prob(B_j \mid D) \leq \ell \cdot 49^{-r}$.
		
		Let $C$ be the event $\set{\textnormal{each } v_i \textnormal{ is coloured } c_i}$. Conditioned on $C$, $B_j$ is increasing in $V_j$ and non-increasing in all other $V_i$. We claim the $B_i$ are negatively correlated on the event $C$. For $k = 2$ this is just the Harris inequality. Fix $k > 2$ and let $S$ be a set of indices: we need to show $\prob(\cap_{i \in S} B_i \mid C) \leq \prod_{i \in S} \prob(B_i \mid C)$. If $\abs{S} \leq 1$, then there is equality. Otherwise let $i_1, i_2 \in S$. Now $B_{i_1} \cap B_{i_2}$ is increasing in $V_{i_1} \cup V_{i_2}$ and non-increasing in all other $V_i$. By induction,
		\begin{equation*}
			\prob(\cap_{i \in S} B_i \mid C) \leq \prob(B_{i_1} \cap B_{i_2} \mid C) \cdot \prod_{i \in S \setminus \set{i_1, i_2}} \prob(B_i \mid C) \leq \prod_{i \in S} \prob(B_i \mid C).
		\end{equation*}
		By symmetry of the colours, $\prob(B_i \mid C) = \prob(B_i \mid D)$ for all $i$ and also $\prob(\cap_{i \in S} B_i \mid C) = \prob(\cap_{i \in S} B_i \mid D)$ for any set of indices $S$. In particular, the $B_i$ are negatively correlated on the event $D$. Applying \cref{Chernoffneg} to $X_1, \dotsc, X_\ell$ gives the result.
	\end{proof}
	
	\begin{lemma}\label{Slemmalargek}
		Let $S$ be a set of vertices of $G$ with $10^{6r} \leq \abs{S} \leq k^{1/2}$. In a uniformly random $k$-colouring of $V(G)$, with failure probability at most $2(e\abs{S}^{-1/2})^{\abs{S}^{1/2}}$, fewer than $3^{-r} \abs{S}$ vertices of $S$ are bad.
	\end{lemma}
	
	\begin{proof}
		Let $A$ be the event that the number of distinct colours on $S$ is at most $\abs{S} - \abs{S}^{1/2}$. We first give an upper bound for $\prob(A)$. The probability that a fixed vertex does not have a unique colour is at most $\abs{S}/k$. If $A$ does occur, then at least $\abs{S}^{1/2}$ vertices of $S$ do not have a unique colour. Hence,
		\begin{equation*}
			\prob(A) \leq \binom{\abs{S}}{\abs{S}^{1/2}} \biggl(\frac{\abs{S}}{k}\biggr)^{\abs{S}^{1/2}} \leq \binom{\abs{S}}{\abs{S}^{1/2}} \abs{S}^{-\abs{S}^{1/2}}.
		\end{equation*}
		If $A$ does not occur, then there is a subset $S' \subset S$ of size $\abs{S} - \abs{S}^{1/2}$ where the vertices are all given different colours. Fix such an $S'$ and let $X$ be the number of bad vertices in $S$ and $X'$ be the number of bad vertices in $S'$. Note that if $X' < 4^{-r} \abs{S'}$, then $X < 4^{-r} \abs{S'} + \abs{S}^{1/2} \leq 4^{-r} \abs{S} + \abs{S}^{1/2} \leq 3^{-r} \abs{S}$.
		
		Let $D$ be the event that all vertices of $S'$ get different colours. By \cref{Chernoffforbad} and the previous paragraph,
		\begin{align*}
			\prob(X \geq \abs{S} \cdot 3^{-r} \mid D) & \leq
			\prob(X' \geq \abs{S'} \cdot 4^{-r} \mid D) \\
			& \leq \prob(X' \geq \abs{S'} \cdot 49^{-r} + \abs{S'} \cdot 4^{-r} /\sqrt{2} \mid D) 
			\leq \exp(-\abs{S'} \cdot 4^{-2r}).
		\end{align*}
		Let $\overline{A}$ be the complement of $A$. Taking a union bound over all $S'$,
		\begin{equation*}
			\prob(\set{X \geq \abs{S} \cdot 3^{-r}} \cap \overline{A}) \leq \binom{\abs{S}}{\abs{S}^{1/2}} \cdot \exp(-\abs{S'} \cdot 4^{-2r}).
		\end{equation*}
		Finally,
		\begin{align*}
			\prob(X \geq \abs{S} \cdot 3^{-r}) & \leq \binom{\abs{S}}{\abs{S}^{1/2}} \bigl(\exp(-\abs{S'} \cdot 4^{-2r}) + \abs{S}^{-\abs{S}^{1/2}}\bigr) \\
			& \leq (e \abs{S}^{1/2})^{\abs{S}^{1/2}} \cdot 2 \abs{S}^{-\abs{S}^{1/2}} = 2(e\abs{S}^{-1/2})^{\abs{S}^{1/2}}. \qedhere
		\end{align*}

	\end{proof}
	
	\subsection{When \texorpdfstring{$k$}{k} is small: \texorpdfstring{$k \leq \Delta^{1/(6r^2)}$}{}}\label{sec:smallk}
	
	Recall that $k = \floor{49(\frac{\Delta}{d + 1})^{1/r}}$ throughout. We need a simple max cut lemma.
	
	\begin{lemma}[Max cut]\label{maxcut}
		Let $G$ be a hypergraph whose edges have size at most $r + 1$ and let $\ell$ be a positive integer. There is a partition $V_1 \cup \dotsb \cup V_{\ell}$ of $V(G)$ such that, for every vertex $x \in V_i$, the number of edges containing $x$ and at least one more vertex from $V_i$ is at most $r \deg(x)/\ell$.
	\end{lemma}
	
	\begin{proof}
		Throughout the proof, vertices $u, v, x$ are distinct. Choose a partition $V_{1} \cup \dotsb \cup V_{\ell}$ of $V(G)$ into $\ell$ parts that minimises
		\begin{equation}\label{mc}
			\sum_{i} \sum_{u, v \in V_{i}} \deg(u, v).
		\end{equation}
		Fix a vertex $x$ and suppose it is in some part $V_{a}$. By minimality, for all $i$,
		\begin{equation*}
			\sum_{u \in V_a} \deg(u, x) \leq \sum_{u \in V_{i}} \deg(u, x),
		\end{equation*}
		or else we could increase \eqref{mc} by moving $x$ to $V_i$. But 
		\begin{equation*}
			\sum_{i} \sum_{u \in V_{i}} \deg(u, x) = \sum_{u \in V(G)} \deg(u, x) \leq r \deg(x),
		\end{equation*}
		and so $\sum_{u \in V_a} \deg(u, x) \leq r \deg(x)/\ell$. This last sum is at least the number of edges containing $x$ and at least one more vertex from $V_a$.
	\end{proof}
	
	Given a large vertex-set $S$ we aim to show that, with high probability, a small proportion of its vertices are bad. We use \cref{maxcut} to split $S$ into parts so that very few edges have two vertices in the same part. Consider an arbitrary part $P$. We will show that, with high probability, a small proportion of the vertices in $P$ are bad. We do this by first revealing the random $k$-colouring on $V(G) - P$. Since $k$ is small, we get strong concentration on the distribution of colours on $V(G) - P$. We then reveal the colouring on $P$ and use this concentration to show that it is unlikely that $P$ has a high proportion of bad vertices.
	
	\begin{lemma}\label{Slemmasmallk}
		Suppose $\Delta \geq 50^{50r^3}$, $k \leq \Delta^{1/r^2}$ and let $S$ be a set of at least $(3k)^{3r} \Delta^{1/(6r)}$ vertices of $G$. With failure probability at most $\Delta^{-6}$, in a uniformly random $k$-colouring of $V(G)$, fewer than $3^{-r} \abs{S}$ vertices of $S$ are bad.
	\end{lemma}
	
	\begin{proof}
		It will be helpful to partition $S$ into sets $P$ such that not too many edges meet $P$ in more than one vertex.  We therefore apply the max cut lemma, \cref{maxcut}, to $G$ with $\ell = r k^r$, and restrict the resulting partition to $S$.  We obtain a partition $\PP$ of $S$ into $r k^r$ parts such that, for every vertex $x \in S$, the number of edges containing $x$ and at least one more vertex from $x$'s part is at most $\deg(x)/k^r$. We say a part $P \in \PP$ is \defn{big} if $\abs{P} \geq \abs{S}/(50 r (3k)^r)$ and is \defn{small} otherwise.
		
		There are $r k^r$ parts in $\PP$, so the number of vertices of $S$ in small parts is at most $\abs{S}/(50 r (3k)^r) \cdot r k^r = 0.02 \cdot 3^{-r} \abs{S}$. Hence, if $3^{-r} \abs{S}$ vertices of $S$ are bad, then at least $0.98 \cdot 3^{-r}$ proportion of the vertices in big parts are bad, so some big part $P$ has at least $0.98 \cdot 3^{-r} \abs{P}$ bad vertices. We now focus on a big part $P \in \PP$ and show that, with failure probability at most $\Delta^{-8}$, at most $0.98 \cdot 3^{-r} \abs{P}$ vertices of $P$ are bad.
		
		For each vertex $x \in P$, let $G'_{x}$ be the $r$-uniform graph on $V(G) - P$, whose edges are those $e$ with $e \cup \set{x} \in E(G)$ (that is, $G'_{x}$ is the link graph of $x$ restricted to $V(G) - P$). Define the $r$-uniform auxiliary (multi)hypergraph $H_P$ to have vertex set $V(G) - P$ and edge set
		\begin{equation*}
			E(H_P) = \bigcup_{x \in P} E(G'_{x}),
		\end{equation*}
		where edges are counted with multiplicity. Let $\phi$ be a uniformly random $k$-colouring of $V(G)$ and $\phi'$ be the restriction of $\phi$ to $V(G) - P$. Reveal $\phi'$ and let $X$ be the number of monochromatic edges of $H_P$, again counted with multiplicity.
		
		We now apply McDiarmid's inequality to show that $X$ concentrates. First note that $e(H_P) \leq \abs{P} \cdot \Delta$ and $\EX(X) = e(H_P) k^{-(r - 1)} \leq \abs{P} \cdot \Delta k^{-(r - 1)}$. For a vertex $v \in V(H_P)$, changing $\phi'(v)$ changes the value of $X$ by at most $\deg_{H_P}(v)$. Now,
		\begin{equation*}
			\sum_{v} \deg_{H_P}(v)^{2} \leq \Delta \sum_{v} \deg_{H_P}(v) = r \Delta e(H_P) \leq r \Delta^2 \abs{P}.
		\end{equation*}
		By McDiarmid’s inequality (\cref{McDiarmid}),
		\begin{align*}
			\prob\Bigl(X \geq \frac{1.1 \cdot \Delta \abs{P}}{k^{r - 1}}\Bigr) 
			\leq \prob\Bigl(X \geq \EX(X) + \frac{0.1 \cdot \Delta \abs{P}}{k^{r - 1}}\Bigr) 
			& \leq \exp\Bigl(- \frac{\abs{P}}{50 r k^{2(r - 1)}}\Bigr) \\
			& \leq \exp\Bigl(- \frac{\abs{S}}{2500 r^2 \cdot 3^r \cdot k^{3r - 2}}\Bigr) \\
			& \leq \exp\bigl(- k^2 \Delta^{1/(6r)}/(2500 r^2)\bigr) \leq \Delta^{-8}/2.
		\end{align*}
		For a vertex $x \in P$, say a colour is \defn{$x$-unhelpful} if there are more than $(49^r - 1) \Delta/k^r$ monochromatic edges of $G'_x$ of that colour. Say $x$ is \defn{unhelpful} if there are more than $0.45 \cdot 3^{-r} k$ $x$-unhelpful colours. Note that if $x$ is unhelpful, then the number of monochromatic edges in $G'_x$ is greater than $0.45 (49^r - 1) \cdot \Delta \cdot 3^{-r} k^{-(r - 1)}$. Hence, if more than $0.48 \cdot 3^{-r} \cdot \abs{P}$ vertices of $P$ are unhelpful, then the number of monochromatic edges in $H_P$ is greater than $1.1 \cdot \Delta \abs{P}/k^{r - 1}$. We have just shown this occurs with probability less than $\Delta^{-8}/2$. Hence, with failure probability at most $\Delta^{-8}/2$, at least $(1 - 0.48 \cdot 3^{-r})\abs{P}$ vertices of $P$ are helpful.
		
		Suppose that at least $(1 - 0.48 \cdot 3^{-r})\abs{P}$ vertices of $P$ are helpful; call the set of helpful vertices $P'$. Now reveal $\phi$ on $P$. For each vertex $x \in P'$, the probability that $x$ gets given an $x$-unhelpful colour is less than $0.45 \cdot 3^{-r}$. Let $Y$ be the number of $x \in P'$ coloured with an $x$-unhelpful colour. For different $x \in P'$, these events are independent (we have already revealed $\phi$ on $V(G) - P$) and so we may couple $Y$ with a random variable $Z \sim \Bin(\abs{P'}, 0.45 \cdot 3^{-r})$ so that $Y \leq Z$. Hence, by the Chernoff bound (\cref{Chernoff}),
		\begin{align*}
			\prob(Y \geq 0.5 \cdot 3^{-r} \abs{P'}) 
			\leq \prob(Z \geq 0.5 \cdot 3^{-r} \abs{P'}) 
			& \leq \prob(Z \geq 1.1 \cdot \EX(Z)) \\
			& \leq \exp(-0.45 \cdot 3^{-r} \abs{P'}/300) \\
			& \leq \exp\bigl(- k^{2r} \Delta^{1/(6r)}/(6000r)\bigr) \leq \Delta^{-8}/2.
		\end{align*}
		Hence, with failure probability at most $\Delta^{-8}/2 + \Delta^{-8}/2 = \Delta^{-8}$, at least $(1 - 0.5 \cdot 3^{-r}) \abs{P'} \geq (1 - 0.98\cdot 3^{-r}) \abs{P}$ vertices $x$ of $P$ are coloured with an $x$-helpful colour.
		
		We now show that if a vertex $x$ is given an $x$-helpful colour, then $x$ will be a good vertex (for $\phi$). Indeed, there are at most $\deg(x)/k^r \leq \Delta/k^r$ edges of $G$ containing $x$ that have at least one more vertex in $P$ and, as $x$ is given an $x$-helpful colour, there are at most $(49^r - 1)\Delta/k^r$ other monochromatic edges containing $x$. Hence, with failure probability at most $\Delta^{-8}$, at least $(1 - 0.98 \cdot 3^{-r}) \abs{P}$ vertices of $P$ are good, that is, at most $0.98 \cdot 3^{-r} \abs{P}$ vertices of $P$ are bad.
		
		Finally, taking a union bound over the big parts shows that the probability some big part $P$ has at least $0.98 \cdot 3^{-r} \abs{P}$ bad vertices is at most $r k^r \Delta^{-8} \leq r \Delta^{-8 + 1/r} \leq \Delta^{-6}$, as required. 
	\end{proof}

	%%%%%%%%%%%%%%%%%%%
	\subsection{Proof of \texorpdfstring{\cref{EachVertexTerrible}}{Lemma 8}} 
	\label{FinalProof}
	
	To prove \cref{EachVertexTerrible} we use the sunflower decompositions given by \cref{sunflower} to show that if a vertex is terrible, then some reasonably large set of vertices $S$ must have at least $3^{-r}$ proportion of its vertices being bad. \Cref{Slemmalargek,Slemmasmallk} show that this is unlikely. 
	
	\begin{proof}[Proof of \cref{EachVertexTerrible}]
		Recall that $\Delta \geq 50^{50r^3}$. Fix a vertex $v$ of $G$ and consider the link graph $G_v$, which is an $r$-uniform hypergraph. Recall that an edge of $G_v$ is \defn{bad} if all its vertices are bad and is \defn{good} otherwise. If $v$ is terrible, then at least $2^{-r} \Delta$ edges of $G_v$ are bad.
		
		First suppose that $k \geq \Delta^{1/(6r^2)}$. By \cref{sunflower}, there are edge-disjoint subgraphs $G_{1}, \dotsc, G_{s}$ of $G_v$ each of which is a sunflower with exactly $\floor{\Delta^{1/(12r^2)}}$ petals and such that $e(G_v - E(G_1 \cup \dotsb \cup G_s)) < r^r \cdot \Delta^{1/(12r)} \leq 6^{-r} \Delta$. Let $G' = G_1 \cup \dotsb \cup G_s$. For each $G_i$, choose a vertex from each petal to form a vertex-set $S_i$. If $v$ is terrible, then the number of bad edges in $G'$ is at least
		\begin{equation*}
			\bigl(2^{-r} - 6^{-r}\bigr) \Delta \geq 3^{-r} \Delta \geq 3^{-r} e(G').
		\end{equation*}
		Hence, if $v$ is terrible, then there is some $i$ for which at least $3^{-r} e(G_i)$ edges of $G_i$ are bad. But, since $S_i$ contains exactly one vertex from each petal of $G_i$, at least $3^{-r} \abs{S_i}$ vertices of $S_i$ are bad. Also, each $S_i$ has size $\floor{\Delta^{1/(12r^2)}} \geq \Delta^{2/(25r^2)} \geq 50^{4r} \geq 10^{6r}$ and $\floor{\Delta^{1/(12r^2)}} \leq k^{1/2}$. Hence, by \cref{Slemmalargek}, at least $3^{-r} \abs{S_i}$ vertices of $S_i$ are bad with probability at most 
		\begin{equation*}
			2(e\abs{S}^{-1/2})^{\abs{S}^{1/2}} \leq 2(e\Delta^{-1/(25r^2)})^{\Delta^{1/(25r^2)}} \leq 2(\Delta^{-1/(50r^2)})^{50^{2r}} \leq 2(\Delta^{-1/(50r^2)})^{400r^2} = 2\Delta^{-8}.
		\end{equation*}
		Taking a union bound over $i$ shows that $v$ is terrible with probability at most $2s\Delta^{-8} \leq \Delta^{-5}$.
		
		Now suppose that $k \leq \Delta^{1/(6r^2)}$. By \cref{sunflower}, there are edge-disjoint subgraphs $G_{1}, \dotsc, G_{s}$ of $G_v$ each of which is a sunflower with at least $\Delta^{1/r}/(6r)$ petals and such that $e(G_v - E(G_{1} \cup \dotsb \cup G_{s})) < 6^{-r} \Delta$. Let $G' = G_{1} \cup \dotsb \cup G_{s}$. For each $G_i$, choose a vertex from each petal to form a vertex-set $S_i$. If $v$ is terrible, then the number of bad edges in $G'$ is at least 
		\begin{equation*}
			\bigl(2^{-r} - 6^{-r}\bigr)\Delta \geq 3^{-r} \Delta \geq 3^{-r} e(G').
		\end{equation*}
		Hence, if $v$ is terrible, then there is some $i$ for which at least $3^{-r} e(G_i)$ edges of $G_i$ are bad and so at least $3^{-r} \abs{S_i}$ vertices of $S_i$ are bad. Now, $(3k)^{3r} \Delta^{1/(6r)} \leq 3^{3r} \Delta^{1/(2r)} \Delta^{1/(6r)} \leq \Delta^{1/r}/(6r) \leq \abs{S_i}$. Hence, by \cref{Slemmasmallk}, at least $3^{-r} \abs{S_i}$ vertices of $S_i$ are bad with probability at most $\Delta^{-6}$. Taking a union bound over $i$ shows that $v$ is terrible with probability at most $s\Delta^{-6} \leq \Delta^{-5}$.
	\end{proof}
	
	\section{Open problems}
	
	As noted in the introduction, Erd\H{o}s and Lov\'{a}sz proved that every $(r + 1)$-uniform hypergraph $G$ with maximum degree at most $\Delta$ has chromatic number $\chi(G) = \OO(\Delta^{1/r})$. \Citet{FM13} improved this to $\OO((\Delta/\log\Delta)^{1/r})$ when $G$ is a linear hypergraph and there have been similar improvements~\citep{CM15,CM16,LP22} when $G$ satisfies other sparsity conditions (such as being triangle-free\footnote{A \defn{triangle} in a hypergrph consists of edges $e, f, g$ and vertices $u, v, w$ such that $u, v \in e$ and $v, w \in f$ and $w, u \in g$ and $\set{u, v, w} \cap e \cap f \cap g = \emptyset$.}).
	
	It would be interesting to know whether logarithmic improvements occur for defective colourings of sparse hypergraphs. \Citet{FM08} showed that there exist $(r + 1)$-uniform linear hypergraphs $G$ with maximum degree $\Delta$ and $\chi(G) = \Omega((\Delta/\log\Delta)^{1/r})$. Consider a $d$-defective $k$-colouring of $G$ (where $d \geq 2$). Each colour class induces a linear $(r + 1)$-uniform hypergraph with maximum degree $d$ and so is $\OO((d/\log d)^{1/r})$-colourable. In particular,
	\begin{equation*}
		k = \Omega\Bigl(\bigl(\tfrac{\Delta}{\log\Delta} \cdot \tfrac{\log d}{d}\bigr)^{1/r}\Bigr).
	\end{equation*}
	We conjecture this is tight.
	\begin{conjecture}
		Every $(r + 1)$-uniform linear hypergraph is $k$-colourable with defect $d \geqslant 2$, where
		\begin{equation*}
			k = \OO\Bigl(\bigl(\tfrac{\Delta}{\log\Delta} \cdot \tfrac{\log d}{d}\bigr)^{1/r}\Bigr).
		\end{equation*}
	\end{conjecture}
	
	Finally, it would be interesting to extend \cref{main} to the list colouring setting.
	
	\bibliographystyle{DavidNatbibStyle}
	\bibliography{thebib}
\end{document}